\documentclass[11pt]{amsart}

\usepackage[a4paper,hmargin=3.5cm,vmargin=4cm]{geometry}
\usepackage{amsfonts,amssymb,amscd,amstext}
\usepackage{graphicx}
\usepackage[dvips]{epsfig}

\usepackage{verbatim}
\usepackage{fancyhdr}
\input xy
\xyoption{all}

\usepackage{times}

\usepackage{enumerate}
\usepackage{titlesec}
\usepackage{mathrsfs}

\titleformat{\section}
{\filcenter\bfseries\large} {\thesection{.}}{0.2cm}{}
\titleformat{\subsection}[runin]
{\bfseries} {\thesubsection{.}}{0.15cm}{}[.]
\titleformat{\subsubsection}[runin]
{\em}{\thesubsubsection{.}}{0.15cm}{}[.]

\usepackage[up,bf]{caption}

\newtheorem{theorem}{Theorem}[section]
\newtheorem{proposition}[theorem]{Proposition}

\newtheorem{lemma}[theorem]{Lemma}
\newtheorem{corollary}[theorem]{Corollary}

\theoremstyle{definition}

\newtheorem{remark}[theorem]{Remark}

\numberwithin{equation}{section}
\numberwithin{figure}{section}

\newcommand\Ncal{\mathcal{N}}

\newcommand\Pcal{\mathcal{P}}

\newcommand\Ascr{\mathscr{A}}

\newcommand\Cscr{\mathscr{C}}

\newcommand\Oscr{\mathscr{O}}

\newcommand\B{\mathbb{B}}
\newcommand\C{\mathbb{C}}
\newcommand\D{\overline{\mathbb D}}
\newcommand\CP{\mathbb{CP}}
\renewcommand\D{\mathbb D}

\newcommand\N{\mathbb{N}}
\renewcommand\P{\mathbb{P}}
\newcommand\R{\mathbb{R}}

\newcommand\Z{\mathbb{Z}}

\renewcommand\c{\mathbb{C}}

\newcommand\igot{\mathfrak{i}}

\renewcommand\igot{\mathfrak{i}}

\newcommand\ggot{\mathfrak{g}}

\renewcommand\imath{\igot}

\newcommand\hra{\hookrightarrow}
\newcommand\lra{\longrightarrow}

\newcommand\wt{\widetilde}

\newcommand\di{\partial}

\newcommand\dist{\mathrm{dist}}

\newcommand\length{\mathrm{length}}

\newcommand\Id{\mathrm{Id}}

\def\dist{\mathrm{dist}}

\def\length{\mathrm{length}}

\newcommand\Span{\mathrm{Span}}

\usepackage{color}

\begin{document}

\fancyhead[LO]{Darboux charts around holomorphic Legendrian curves}
\fancyhead[RE]{A.\ Alarc\'on and F.\ Forstneri\v c} 
\fancyhead[RO,LE]{\thepage}

\thispagestyle{empty}

\vspace*{1cm}
\begin{center}
{\bf\LARGE Darboux charts around holomorphic Legendrian curves and applications}

\vspace*{0.5cm}

{\large\bf  Antonio Alarc{\'o}n and Franc Forstneri{\v c}} 
\end{center}


\vspace*{1cm}

\begin{quote}
{\small
\noindent {\bf Abstract}\hspace*{0.1cm}
In this paper we find a holomorphic Darboux chart around any immersed noncompact holomorphic Legendrian curve
in a complex contact manifold $(X,\xi)$. By using such a chart, we show that every holomorphic Legendrian immersion $R\to X$ 
from an open Riemann surface can be approximated on relatively compact subsets of $R$ by 
holomorphic Legendrian embeddings, and every holomorphic Legendrian immersion $M\to X$ 
from a compact bordered Riemann surface is a uniform limit of topological embeddings $M\hra X$ such that  $\mathring M\hra X$ 
is a complete holomorphic Legendrian embedding. We also establish a contact neighborhood theorem for
isotropic Stein submanifolds in complex contact manifolds.

\vspace*{0.2cm}

\noindent{\bf Keywords}\hspace*{0.1cm} complex contact manifold, Darboux chart, Legendrian curve, 
isotropic submanifold

\vspace*{0.1cm}

\noindent{\bf MSC (2010):}\hspace*{0.1cm} 53D10; 32E30, 32H02, 37J55}

\vspace*{0.1cm}
\noindent{\bf Date: \rm February 2, 2017. This version: \today}

\end{quote}


\section{Introduction and main results} 
\label{sec:intro}

A {\em complex contact manifold} is a pair $(X,\xi)$, where $X$ is a complex manifold of 
(necessarily) odd dimension $2n+1\ge 3$ and $\xi$ is a completely noninvolutive holomorphic hyperplane subbundle 
(a {\em contact subbundle}) of the holomorphic tangent bundle $TX$. 
Locally near any point of $X$, a contact subbundle  is the kernel of a holomorphic $1$-form 
$\alpha$ satisfying $\alpha \wedge(d\alpha)^n \neq 0$; such $\alpha$ is called a {\em holomorphic contact form}. 
(Globally we have $\xi=\ker \alpha$ for a holomorphic $1$-form with coefficients in the
line bundle $\nu=TX/\xi$; see Sec.\ \ref{sec:CNT}.)
By a fundamental theorem of Darboux \cite{Darboux1882CRAS} from 1882, there are local 
coordinates $(x_1,y_1,\ldots, x_n,y_n,z)$ at any point of $X$ in which 
the contact structure $\xi$ is given by the {\em standard contact form}
\begin{equation}\label{eq:alpha0}
	\alpha_0= dz + \sum_{j=1}^n x_j \, dy_j. 
\end{equation}
The proof of Darboux's theorem given by Moser \cite{Moser1965TAMS} (see also \cite[p.\ 67]{Geiges2008}) 
easily adapts to the holomorphic case; see \cite[Theorem A.2]{AlarconForstnericLopez2017CM}. 
This shows that a holomorphic contact structure has no local invariants, and hence all interesting problems are 
of global nature.

Let $(X,\xi)$ be a complex contact manifold of dimension $2n+1\ge 3$. A smooth immersed submanifold $f\colon R\to X$
is said to be {\em isotropic} if 
\[
	\text{$df_x(T_x R)\subset \xi_{f(x)}$\ \ holds for all $x\in R$.}
\]
If $\xi=\ker\alpha$, this is equivalent to $f^*\alpha=0$. The contact condition implies $\dim_\R R\le 2n$;
the immersion is said to be {\em Legendrian} if $R$ is of maximal dimension $\dim_\R R=2n$. 
(See Sec.\ \ref{sec:CNT} for more details.) It is easily seen that the image of a smooth Legendrian immersion 
is necessarily a complex submanifold of $X$ (see Lemma \ref{lem:Legendrian}).
Since we shall mainly consider the case when $R$ is an open Riemann surface and $f\colon R\to X$
is an isotropic holomorphic curve, we will use the term {\em Legendrian curve} also when $\dim X\ge 5$,
with the exception of Sec.\ \ref{sec:CNT} where we consider also higher dimensional complex isotropic submanifolds.

In this paper we prove that there exists a holomorphic Darboux chart around any 
immersed noncompact holomorphic Legendrian curve, and also around some higher dimensional
isotropic Stein submanifolds, in an arbitrary complex contact manifold. 
The following is our first main result; it is proved in Section \ref{sec:normal}. 

%
%

\begin{theorem}\label{th:normal}
Let $(X,\xi)$ be a  complex contact manifold of dimension $2n+1\ge 3$. Assume that 
$R$ is an open Riemann surface, $\theta$ is a nowhere vanishing holomorphic $1$-form on $R$,
and $f\colon R\to X$ is a holomorphic Legendrian immersion.
Then there are an open neighborhood $\Omega\subset R\times\C^{2n}$ of $R\times \{0\}^{2n}$
and a holomorphic immersion $F\colon \Omega \to X$ (embedding if $f$ is an embedding) such that 
$F|_{R\times \{0\}^{2n}}=f$ and  the contact structure $F^*\xi$ on $\Omega$ is given by the contact form
\begin{equation}\label{eq:normal}
	\alpha= dz- y_1\theta - \sum_{i=2}^n y_i dx_i,
\end{equation}
where 
$(x_2,\ldots, x_n,y_1,\ldots,y_n,z)$ are complex coordinates on  $\C^{2n}$.
\end{theorem}

Recall that any holomorphic line bundle on an open Riemann surface $R$ is holomorphically trivial
according to the Oka-Grauert principle (see \cite[Theorem 5.3.1, p.\ 190]{Forstneric2011}); in particular,
$R$ admits a nowhere vanishing holomorphic $1$-form. Furthermore, by the Gunning-Narasimhan theorem 
(see \cite{GunningNarasimhan1967MA}  or  \cite[Corollary 8.12.2, p.\ 386]{Forstneric2011})
there exist plenty of holomorphic immersions $x_1\colon R\to\c$, i.e., holomorphic functions without critical 
points. Taking $\theta=dx_1$ and replacing $z$ by $z+\sum_{j=1}^n x_j y_j$, 
the normal form \eqref{eq:normal} changes to
\begin{equation}\label{eq:normal2}
	\alpha_0 = dz+ \sum_{i=1}^n x_i dy_i.
\end{equation}
This is formally the same as \eqref{eq:alpha0}, the difference being that $x_1$ is now 
a holomorphic immersion $R\to\C$ (and not just a local coordinate function), 
and the normal form \eqref{eq:normal2} is valid 
globally in a tube around the immersed Legendrian curve $f\colon R\to X$.

The existence of global holomorphic Darboux charts, given by Theorem \ref{th:normal}, 
has many applications, some of which are presented in the sequel.
It shows that the contact structure has no local invariants along a noncompact holomorphic Legendrian curve,
and any such curve extends to a local complex Legendrian submanifold of maximal dimension $n$;
in the Darboux chart \eqref{eq:alpha0} this extension is provided by $\{y=0,\ z=0\}=R\times \C^{n-1}_{x_2,\ldots,x_n}$. 
In the case $\dim X=3$, we see that all small Legendrian perturbations of the Legendrian curve 
$R\times \{0\}^2$ in $(R\times \C^2,\alpha)$ with  $\alpha = dz-y\theta$
are of the form $z=g(x)$ and $y=dg(x)/\theta(x)$, where $x\in R$ and $g$ is a holomorphic function on $R$.

%
%
Before proceeding, we wish to briefly address the question
that might be asked by the reader at this point: {\em How many complex contact manifolds are there?}
Examples and constructions of such manifolds can be
found in the papers \cite{Beauville1998,Beauville2011,Buczynski2010DM,Kebekus-etal2000,LandsbergManivel2007AJM,
LeBrun1995IJM,LeBrunSalamon1994IM,Kobayashi1959,Wolf1965,Ye1994IM}, among others.
Many of these constructions mimic those in smooth contact geometry
(for the latter, see e.g.\ Cieliebak and Eliashberg \cite{CieliebakEliashberg2012} and Geiges \cite{Geiges2008}).
On the other hand, some of the constructions in the complex world have no analogue in the real one.

If $\xi$ is a holomorphic contact structure on a complex manifold $X^{2n+1}$, then its normal line bundle $\nu=TX/\xi$ 
satisfies $\nu^{\otimes(n+1)}=K_X^{-1}$ where $K_X=\det(T^*X)$ is the canonical bundle of $X$
(see \cite{LeBrun1995IJM,LeBrunSalamon1994IM} and \eqref{eq:nu} below for a discussion of this topic). 
Assuming that $X$ is compact and $\nu\to X$ is a holomorphic line bundle satisfying this condition, 
the space of all contact structures on $X$ with normal bundle isomorphic to $\nu$ is a (possibly empty) 
connected complex manifold (see LeBrun \cite[p.\ 422]{LeBrun1995IJM}). It follows from 
Gray's theorem \cite{Gray1959AM} that all these contact structures are contactomorphic to each other. 
If $X$ is simply connected, it admits at most one isomorphism class of the root
$K_X^{-1/(n+1)}$, and hence at most one complex contact structure up 
to contactomorphisms (see \cite[Proposition 2.3]{LeBrun1995IJM}). 
Only two general constructions of compact complex contact manifolds are known: 

\begin{enumerate}
\item 
The projectivized cotangent bundle $\P(T^*Z)$ of any complex manifold $Z$ of dimension at least 2.
Recall that $T^*Z$ carries a tautological $1$-form $\eta$,
given in any local holomorphic coordinates $(z_1,\ldots,z_n)$ on $Z$ and associated fibre coordinates 
$(\zeta_1,\ldots,\zeta_n)$ on $T_z^*Z$ by $\eta= \sum_{j=1}^n \zeta_j dz_j$. 
Considering $(\zeta_1,\ldots,\zeta_n)$ as  projective coordinates on $\P(T_z^*Z)$, we get the contact 
structure $\xi=\ker\eta$ on $X=\P(T^*Z)$.
\vspace{1mm}
\item 
Let $G$ be a simple complex Lie group with the Lie algebra $\ggot$. The adjoint action of $G$ on the 
projectivization $\P(\ggot)$ of $\ggot$ has a unique closed orbit $X_{\ggot}$ which is contained in 
the closure of every other orbit; this orbit $X_{\ggot}$ is a contact Fano manifold. 
See the papers by Boothby \cite{Boothby1961}, Wolf \cite{Wolf1965} and Beauville \cite{Beauville1998,Beauville2007}.
The simplest example of this type is the contact structure on the projective spaces $\CP^{2n+1}$.
\end{enumerate}

%
%
It is conjectured that any projective contact manifold is of one of these two types 
(see Beauville \cite[Conjecture 6]{Beauville2011}). 
For projective threefolds, this holds true according to Ye \cite{Ye1994IM}.
Demailly proved \cite[Corollary 2]{Demailly2002} that if a compact K{\"a}hler manifold $X$ admits a 
contact structure, then its canonical bundle $K_X$ is not pseudo-effective (and hence not nef), 
and in particular the Kodaira dimension of $X$ equals $-\infty$. 
(The latter fact was also shown by  Druel \cite[Proposition 2]{Druel1998}.)
If in addition $b_2(X) = 1$  then $X$ is projective and hence $K_X$, not being  pseudo-effective,
is negative, i.e., $X$ is a Fano manifold \cite[Corollary 3]{Demailly2002}. 
Together with the results by Kebekus et al. \cite[Theorem 1.1]{Kebekus-etal2000}
it follows that  a projective contact manifold is Fano with $b_2=1$ or of type (1);
see \cite[Corollary 4]{Demailly2002} and Peternell \cite[Theorem 2.9]{Peternell2001}. 
Since a homogeneous Fano contact manifold is known to be of type (2), the 
above conjecture reduces to showing that every Fano contact manifold is homogeneous. 

The situation is more flexible on noncompact complex manifolds, and especially on Stein manifolds. 
A holomorphic $1$-form $\alpha$ which is contact at some point of $X^{2n+1}$ is contact
in the complement of a closed complex hypersurface (possibly empty) given by the vanishing of the $(2n+1)$-form 
$\alpha\wedge (d\alpha)^n$; this holds for a generic holomorphic $1$-form on a Stein manifold.
If $(Z,\omega)$ is a holomorphic symplectic manifold and $V$ is a holomorphic 
vector field on $Z$ satisfying $L_V\omega=\omega$ (such $V$ is called a {\em Liouville vector field}; here, $L_V$ denotes 
the Lie derivative), then the restriction of the $1$-form $\alpha=i_V\omega=\omega(V,\cdotp)$ to any 
smooth complex hypersurface $X\subset Z$ transverse to $V$ is a contact form on $X$.
For example, letting $\omega=\sum_{j=1}^{n+1} dz_{2j-1}\wedge dz_{2j}$ (the standard symplectic form on $\C^{2n+2}$)
and $V=\sum_{j=1}^{2n+2} z_j\di_{z_j}$ yields the standard contact structure on any hyperplane 
$H_k=\{z_k=1\}$ $(k=1,\ldots,2n+2$), and hence it induces a complex contact structure on $\CP^{2n+1}$.


We now present the applications of Theorem \ref{th:normal} obtained in this paper.

In Section \ref{sec:local} we prove the following general position result
whose proof combines Theorem \ref{th:normal} and the arguments from 
\cite[proof of Lemma 4.4]{AlarconForstnericLopez2017CM}.

\begin{theorem} \label{th:embedded}
Let $X$ be a complex contact manifold.
Every holomorphic Legendrian immersion $f\colon R\to X$ from an open Riemann surface
can be approximated, uniformly on any relatively compact subset $U\Subset R$, by 
holomorphic Legendrian embeddings $\tilde f \colon U\hra X$.
\end{theorem}

In general one cannot approximate a holomorphic Legendrian  immersion $f\colon R\to X$ uniformly on compacts in $R$ by
holomorphic Legendrian embeddings $R\hra X$ of the whole Riemann surface. 
For example, if $f\colon \C\to\C^3$ is a proper holomorphic immersion with a single double point, 
then by choosing $X\subset \C^3$ to be an open neighborhood of $f(\C)$ which is very thin near infinity
we can ensure that $f(\C)$ is the only nonconstant complex line in $X$ up to a reparametrization. 
Approximation of Legendrian embeddings by global ones is possible in the model contact space 
$(\C^{2n+1},\alpha_0)$; moreover, there exist {\em proper} holomorphic Legendrian embeddings
$R\hra \C^{2n+1}$ from any open Riemann surface $R$ 
(see Alarc\'on et al. \cite[Theorem 1.1]{AlarconForstnericLopez2017CM}). The latter result also holds 
for the special linear group $SL_2(\c)$ endowed with its standard contact structure (see 
Alarc\'on \cite{Alarcon2016SL2C}). On the other hand, there exist examples of 
(Kobayashi) {\em hyperbolic} complex contact structures on $\C^{2n+1}$ for any $n\ge 1$ 
(see Forstneri{\v c} \cite{Forstneric2016Hyp}); in particular, these structures do not admit any 
nonconstant holomorphic Legendrian maps from $\C$ or $\C^*$. 

The proof of Theorem \ref{th:embedded} also provides local deformation theory of noncompact 
holomorphic Legendrian curves. In particular, the space of Legendrian deformations of a 
holomorphic Legendrian curve normalized by a bordered Riemann surface is an infinite dimensional 
complex Banach manifold (see Remark \ref{remark:deformation}).


Another application of Theorem \ref{th:normal} is that we can uniformly approximate 
a holomorphic Legendrian curve with smooth boundary in an arbitrary complex contact manifold $(X,\xi)$ 
by complete holomorphic Legendrian embeddings bounded by Jordan curves.
In order to formulate this result, we need to recall the following notions.
Assume that $\ggot$ is a Riemannian metric on $X$. An immersion $f\colon R \to X$ is said to be 
{\em complete} if the induced metric $f^*\ggot$ on $R$ is a complete metric.  
A compact bordered Riemann surface, $M$, is the same thing as a compact smoothly bounded domain
in an open Riemann surface $R$. A map $f\colon M\to  X$ from such a domain is said to be holomorphic 
if it extends to a holomorphic map on an open neighborhood of $M$in $R$. 

%
%
\begin{theorem}\label{th:main}
Let $(X,\xi)$ be a complex contact manifold, and let $M$ be a compact bordered Riemann surface.
Every holomorphic Legendrian immersion $f_0\colon M\to X$ can be approximated uniformly on $M$ 
by topological embeddings $f\colon M\to X$ such that $f|_{\mathring M} \colon \mathring M\to X$ is a complete 
holomorphic Legendrian embedding.
\end{theorem}

Since $f_0(M)$ is a compact subset of $X$, the notion of completeness of complex curves $f\colon M\to X$
uniformly close to $f_0$ is independent of the choice of a metric on $X$.

Theorem \ref{th:main}  may be compared with the results on the Calabi-Yau problem in the theory of 
conformal minimal surfaces in $\R^n$ and null holomorphic curves in $\C^n$; 
see Alarc{\'o}n et al.\ \cite{AlarconDrinovecForstnericLopez2015PLMS,AlarconForstnericLopez2017Memoirs} 
and the references therein for recent developments on this subject.

Theorem \ref{th:main} is proved in Section \ref{sec:proof-main}.
The special case with $(X,\xi)$ the model contact space $(\C^{2n+1},\alpha_0)$ 
(see \eqref{eq:alpha0}) was obtained in \cite[Theorem 1.2]{AlarconForstnericLopez2017CM}.
The Darboux charts, furnished by Theorem \ref{th:normal}, make it possible to extend
this result to any complex contact manifold.

In Section \ref{sec:CNT} we consider isotropic complex submanifolds $M$ of higher dimension in a
complex contact manifold, and we prove a contact neighborhood theorem in the case when
$M$ is a Stein submanifold (see Theorem \ref{th:contactnbd}). 
In particular, we obtain the following result.

\begin{theorem}\label{th:cor:contactnbd}
Let $(X_i,\xi_i)$ $(i=0,1)$ be complex contact manifolds of the same dimension.
If $M_i \subset X_i$ $(i=0,1)$ are biholomorphic Legendrian (i.e., isotropic and of maximal dimension) 
Stein submanifolds such that $\nu_i=TX_i/\xi_i$ is trivial over $M_i$ for $i=0,1$, 
then $M_0$ and $M_1$ have holomorphically contactomorphic neighborhoods.
\end{theorem}

Moreover, in the special case when the complex isotropic submanifold $M\subset (X,\xi)$ is Stein and
contractible, we find a Darboux chart around $M$ similar to those furnished by Theorem \ref{th:normal}; 
see Theorem \ref{th:Darboux2}.

It seems that the results in this paper are the first of their kind in the holomorphic case. On the other hand,
contact neighborhood theorems of isotropic submanifolds are well known in the smooth case. 
For example, two smooth diffeomorphic isotropic submanifolds 
with isomorphic conformal symplectic normal bundles have contactomorphic neighborhoods
(see Geiges \cite[Theorem 2.5.8]{Geiges2008}). In particular,  diffeomorphic closed 
Legendrian submanifolds (i.e., isotropic submanifolds of maximal dimension) have 
contactomorphic neighborhoods (see \cite[Corollary 2.5.9]{Geiges2008}).
For example, if $S^1\subset (X^3,\xi)$ is a Legendrian knot in a smooth contact 3-manifold,
then with a coordinate $x$ along $S^1$ and coordinates $y,z$ in slices transverse to $S^1$,
the contact form 
\[
	\cos x\cdotp dy - \sin x\cdotp dz
\]
provides a model for a contact neighbourhood of $S^1$ in $X$ (see \cite[Example 2.5.10]{Geiges2008}).
By  the proof of Theorem \ref{th:normal} we can also get a contact neighborhood with the form 
$dz-ydx$. However, a crucial difference appears  between the real and the complex case: 
there is no smooth immersion $S^1\to\R$, so $dx$ only has the meaning as a nonvanishing $1$-form on $S^1$. 
In particular, the $1$-form $dz+xdy$ is not contact for any smooth function $x\colon S^1\to\R$, and there are no
smooth contact neighborhoods  \eqref{eq:normal2} of a smooth Legendrian knot.

It is natural to ask what could be said about contact neighborhoods of compact holomorphic Legendrian curves
and, more generally, of higher dimensional compact isotropic  complex submanifolds.
According to Bryant \cite[Theorem G]{Bryant1982JDG} (see also Segre \cite{Segre1926}),
every compact Riemann surface embeds as a complex Legendrian curve in $\CP^3$.
The first question to answer is which closed Legendrian curves in $\CP^3$ admit Darboux type neighborhoods.
One major obstacle is that the tubular neighborhood theorem  fails in general for compact complex
submanifolds.  In another direction, the deformation theory of certain compact complex Legendrian submanifolds
has been studied by Merkulov \cite{Merkulov1994MRL} by using  Kodaira's deformation
theory approach \cite{Kodaira1962AM}. He showed that a compact complex Legendrian submanifold $M$ of $(X,\xi)$ with
$H^1(M,\xi)=0$ is contained in a complete analytic family of compact complex Legendrian submanifolds of $X$.


\section{Normal form of a contact structure along an immersed Legendrian curve}
\label{sec:normal}

In this section we prove Theorem \ref{th:normal}. We shall repeatedly use the following known lemma;
we include a sketch of proof for the sake of completeness.

\begin{lemma}\label{lem:GL}
Let $R$ be an open Riemann surface, and let $A$ be a holomorphic $m\times p$ matrix-valued
function on $R$, with $1\le m<p$, which has maximal rank $m$ at every point of $R$. 
Then there exists a holomorphic map $B\colon R\to GL_p(\C)$
such that $A(x)\cdotp B(x)=(I_m,0)$ holds for all $x\in R$, where $I_m$ is the $m\times m$ identity matrix. 
\end{lemma}
\begin{proof}
Denote the rows of $A$ by $a^j$ for $j=1,\ldots,m$; these are holomorphic maps $R\to \C^p$ 
such that the vectors $a^j(x)$ are linearly independent at every point $x\in R$. 
We must find holomorphic maps $a^{m+1},\ldots,a^p \colon R\to\C^p$
such that the matrix 
function $\wt A\colon R\to \C^{p\times p}$ with the rows $a^1,a^2, \ldots,a^p$
is invertible at each point; then $B=\wt A^{-1}$ satisfies the lemma. 

Recall that every holomorphic vector bundle on an open Riemann surface $R$ is trivial by the Oka-Grauert principle
(see \cite[Theorem 5.3.1, p.\ 190]{Forstneric2011}), 
and every holomorphic vector subbundle $E'$ of a holomorphic vector bundle $E$ over a
Stein manifold splits $E$, i.e., we have $E=E'\oplus E''$ where $E''$ is another holomorphic vector subbundle of $E$
(this follows from Cartan's Theorem B, see \cite[Corollary 2.4.5, p.\ 54]{Forstneric2011}). Let $E=R\times \C^p$,
and let $E'\subset E$ be the holomorphic rank $m$ subbundle spanned by the rows of the matrix 
$A(x)$ at each point $x\in R$. Then $E=E'\oplus E''$ where $E''$ is a trivial bundle of rank $p-m$; 
thus it is generated by $p-m$ global holomorphic sections $a^{m+1},\ldots,a^p\colon R\to \C^p$. This proves
the lemma.
\end{proof}

\begin{remark}\label{rem:lemGL}
The conclusion of Lemma \ref{lem:GL} holds for any Stein manifold $R$ on which every complex
vector bundle is topologically trivial; for instance, on a contractible Stein manifold.
Indeed, by the Oka-Grauert principle 
it follows that every holomorphic vector bundle on $R$ is holomorphically trivial, and hence
the proof of Lemma \ref{lem:GL} applies verbatim. 
\end{remark}

%
%
\begin{proof}[Proof of Theorem \ref{th:normal}]
The normal bundle of the immersion $f\colon R\to X$ is a holomorphic vector bundle of rank $2n$
over $R$, hence a trivial bundle by the Oka-Grauert principle (see \cite[Theorem 5.3.1]{Forstneric2011}).
By the Docquier-Grauert  tubular neighborhood theorem (see \cite[Theorem 3.3.3]{Forstneric2011}),
there are a Stein open neighborhood $\Omega \subset R\times \C^{2n}$ of $R\times \{0\}^{2n}$ 
and a holomorphic immersion $F\colon \Omega \to X$ with $F|_{R\times \{0\}^{2n}}=f$. 
Furthermore, $\Omega$ can be chosen to have convex fibers, so it is homotopy equivalent to $R$.
Hence, every holomorphic vector bundle on $\Omega$ is holomorphically trivial by the 
Oka-Grauert principle. 
In particular, the complex line bundle $\nu=T\Omega/F^*\xi$ is trivial,  
and the quotient projection $\beta\colon T\Omega\to \nu\cong \Omega\times\C$ with 
$\ker\beta=F^*\xi$ is a holomorphic $1$-form on $\Omega$ defining the contact structure $F^*\xi$.
(Compare with \eqref{eq:normalbundle}.)  The contact condition is that 
\[
	\beta \wedge(d\beta)^n \neq 0.
\]
We shall find a holomorphic change of coordinates in a neighborhood 
of $R\times \{0\}^{2n}$ which fixes $R\times \{0\}^{2n}$ pointwise and reduces $\beta$ to the form \eqref{eq:normal}. 
To simplify the notation, the neighborhood in question will always be called $\Omega$, but the reader
should keep in mind that it is allowed to shrink around $R\times \{0\}^{2n}$ during the proof. 

Let $x$ denote points in $R$, and let $\zeta=(\zeta_1,\ldots,\zeta_{2n})$ be complex coordinates on $\C^{2n}$. 
Along $R\times \{0\}^{2n}=\{\zeta=0\}$ we have that
\[
	\beta(x)=\sum_{j=1}^{2n} a_j(x) d\zeta_j,\quad x\in R
\]
for some holomorphic functions $a_j\in\Oscr(R)$ without common zeros
(since $\ker\beta(x)$ is the contact hyperplane at $(x,0)\in R\times \{0\}^{2n}$). The $1$-form
$\theta$ does not appear in the above expression since $R\times \{0\}^{2n}$ is a
$\beta$-Legendrian curve. Let $a=(a_1,\ldots,a_{2n})\colon R \to \C^{2n}\setminus \{0\}$. 
We introduce new coordinates $\zeta'=B(x)^{-1}\zeta$, where the holomorphic map
$B\colon R \to GL_{2n}(\C)$ satisfies $a(x)\cdotp B(x)=(1,0,\ldots,0)$
for all $x\in R$; such $B$ exists by Lemma \ref{lem:GL}. Dropping the primes, this transforms $\beta$ 
along $R\times \{0\}^{2n}$ to the constant $1$-form $d\zeta_1$. 
Geometrically speaking, this amounts to rotating the contact plane $\xi_x$ for 
$x\in R\times \{0\}^{2n}$ to the constant position given by $d\zeta_1=0$. 
Denoting the variable $\zeta_1$ by $z$, we have $\beta=dz$ at all points of $R\times \{0\}^{2n}$.

We now consider those terms in the Taylor expansion of $\beta$ along $R\times \{0\}^{2n}$
which give a nontrivial contribution to the coefficient function of the $(2n+1)$-form  $\beta \wedge(d\beta)^n$. 
Since the coefficient of $dz$ equals $1$ on $R\times \{0\}^{2n}$, it is a nowhere vanishing holomorphic 
function in a neighborhood of this set, and we simply divide $\beta$ by it. We thus have
%
%
\begin{equation}\label{eq:Taylor1}
	\beta=dz+ \biggl(\sum_{j=2}^{2n} b_j(x)\zeta_j \biggr) \theta(x) 
	+ \sum_{j,k=2}^{2n} c_{j,k}(x)\zeta_k\, d\zeta_j + \tilde \beta,
\end{equation}
where the coefficients $b_j$ and $c_{j,k}$ are holomorphic functions on $R$.
The $1$-form $\tilde \beta$ (the remainder) contains all terms $\zeta_jd\zeta_j$, terms whose coefficients are of order 
$\ge 2$ in the variables  $\zeta_2,\ldots,\zeta_{2n}$, or terms that contain the $z$ variable; 
such terms disappear in $\beta \wedge(d\beta)^n$ at all points of $R \times \{0\}^{2n}$.

We claim that the functions $b_2,\ldots,b_{2n}$  in \eqref{eq:Taylor1} have no common zeros in $R$. 
Indeed, at a common zero $x_0\in R$ of these functions, the form $d\beta$ at the point $(x_0,0)$
does not contain the term $\theta(x_0)$ and hence $\beta \wedge(d\beta)^n$ vanishes, a contradiction. 
Write $\zeta'=(\zeta_2,\ldots,\zeta_{2n})$. Applying Lemma \ref{lem:GL} with the
row matrix $b=(b_2,\ldots,b_{2n})\colon R\to\C^{2n-1} \setminus\{0\}$
gives a holomorphic change of coordinates of the form 
\[
	(x,z,\zeta')\mapsto (x,z,B(x)\zeta'), \quad B(x)\in GL_{2n-1}(\C)
\]
such that the coefficient of $\theta$ becomes $-\zeta_2$, and hence 
\begin{equation}\label{eq:Taylor2}
	\beta=dz - \zeta_2 \theta +  \sum_{j,k=2}^{2n} c_{j,k}(x)\zeta_k \, d\zeta_j + \tilde \beta
\end{equation}
for some new coefficients $c_{j,k}$. 
%
%
%
%
Note that $(d\beta)^n$ contains the factor $d(\zeta_2\theta)=d\zeta_2 \wedge\theta$
(since a nontrivial differential in the $R$-direction does not appear in any other way).
Hence, the term with $d\zeta_2$ and all terms containing $\zeta_2d\zeta_j$ with $j> 2$ 
in \eqref{eq:Taylor2} can be placed into the remainder $\tilde\beta$
since they do not contribute to $(d\beta)^n$. Renaming the variable $\zeta_2$ by $y_1$ we thus have
%
%
\begin{equation}\label{eq:Taylor3}
	\beta=dz -  y_1 \theta + \sum_{j=3}^{2n}  \left( \sum_{k=4}^{2n} c_{j,k}(x)\zeta_k \right)  d\zeta_j  + \tilde \beta.
\end{equation}
If $n=1$ (i.e., $\dim X=3$), we are finished with the first part of the proof and proceed to the second part
given below.  

Assume now that $n>1$. We begin by eliminating the variable $\zeta_3$ 
from the coefficients of the differentials $d\zeta_4,\ldots,d\zeta_{2n}$ by the shear 
%
%
\[
	z' = z+ \sum_{j=4}^{2n} c_{j,3}(x)\zeta_3 \zeta_j.
\]
This ensures that  the functions $c_{3,k}$ in the coefficient of $d\zeta_3$ in \eqref{eq:Taylor3} have no common zeros 
on $R$ (since at such point $d\beta$ would not contain $d\zeta_3$).
Applying Lemma \ref{lem:GL} we change the coefficient of $d\zeta_3$ to $-\zeta_4$ 
by a linear change of the variables $\zeta_4,\ldots, \zeta_{2n}$ with a holomorphic dependence on $x\in R$. 
Set $x_2=\zeta_3$ and $y_2=\zeta_4$. By the same argument as in the previous step, 
we can move the term with $dy_2=d\zeta_4$, as well as  all terms containing $y_2=\zeta_4$ in the subsequent differentials 
$d\zeta_5,\ldots,d\zeta_{2n}$, to the remainder $\tilde \beta$. This gives
\begin{equation}\label{eq:Taylor4}
	\beta=dz - y_1\theta - y_2dx_2 + \sum_{j,k=5}^{2n} c_{j,k}(x)\zeta_k\, d\zeta_j + \tilde \beta.
\end{equation}
It is clear that this process can be continued, and in finitely many steps we obtain
\[
	\beta = dz- y_1\theta - \sum_{i=2}^n y_i dx_i + \tilde \beta = \alpha +\tilde \beta,
\]
where $\alpha$ is the normal form \eqref{eq:normal}. 

We now complete the proof by applying Moser's method
\cite{Moser1965TAMS} in order to get rid of the remainder $\tilde \beta$.
Consider the following family of holomorphic $1$-forms on $\Omega$:
\[
	\alpha_t=\alpha+t(\beta-\alpha)= \alpha + t\tilde \beta, \quad  t\in[0,1].
\]
Note that $\alpha_0=\alpha$, $\alpha_1=\beta$, and for all $t\in [0,1]$ we have
\[
	\alpha_t=\alpha \quad  \text{and} \quad
	\alpha_t\wedge (d\alpha_t)^{n} = \alpha\wedge(d\alpha)^n\ \ 
	\text{on}\ \ R\times \{0\}^{2n}.
\]
The second identity holds because, by the construction, $\tilde \beta$ contains only
terms which do not contribute to $\alpha_t\wedge (d\alpha_t)^{n}$. Hence, $\alpha_t$ is a contact form in a  
neighborhood of $R\times \{0\}^{2n}$, still denoted $\Omega$, determining a contact structure $\xi_t=\ker\alpha_t$ 
for every $t\in [0,1]$,  and $\dot\alpha_t=\beta-\alpha$ vanishes on $R\times \{0\}^{2n}$. (The dot indicates the $t$-derivative.)
We shall find a time-dependent holomorphic vector field $V_t$ on a neighborhood of 
$R\times \{0\}^{2n}$ that vanishes on $R\times \{0\}^{2n}$ and whose flow $\phi_t$ satisfies
\begin{equation}\label{eq:flow}
	\phi_t^* \alpha_t = \alpha,\quad  t\in[0,1]
\end{equation}
and the initial condition
\begin{equation}\label{eq:initialcond}
	\phi_t(x,0,\ldots,0)=(x,0,\ldots,0),\quad x\in R,\ t\in[0,1].
\end{equation}
At time $t=1$ we shall then get $\phi_1^*\beta = \alpha$ in an open neighborhood of $R\times \{0\}^{2n}$,
thereby completing the proof of the theorem. 

Let $\Theta_t$ denote the Reeb vector field of contact form  $\alpha_t$
(see \cite[p.\ 5]{Geiges2008}), i.e., the unique holomorphic vector field satisfying the conditions
\[
	\Theta_t \,\rfloor\, \alpha_t = \alpha_t(\Theta_t) =1\quad \text{and}\quad  
	\Theta_t \,\rfloor\, d\alpha_t= \langle d\alpha_t, \Theta_t\, \wedge\,\cdotp\rangle =0.
\]
A vector field $V_t$ whose flow  satisfies conditions \eqref{eq:flow}, \eqref{eq:initialcond} is sought in the form 
\begin{equation}\label{eq:Vt}
	V_t=h_t \Theta_t + Y_t,  \quad  t\in [0,1] 
\end{equation}
where $h_t$ is a smooth family of holomorphic functions and $Y_t$ is a smooth family of 
holomorphic vector fields tangent to $\ker\alpha_t$ on a neighborhood of $R \times \{0\}^{2n}$. Then, 
\[
	V_t \,\rfloor\, \alpha_t = h_t\quad \text{and}\quad V_t\,\rfloor\, d\alpha_t = Y_t \,\rfloor\, d\alpha_t.
\]
Differentiating the equation \eqref{eq:flow} on $t$ gives
\[
	\phi^*_t\bigl(\dot\alpha_t + L_{V_t}\alpha_t \bigr)=0,
\]
where $L$ denotes the Lie derivative. By using Cartan's formula
$L_V\alpha=d(V\rfloor \alpha) + V\rfloor\, d\alpha$ we see that $V_t$ must satisfy the equation
\begin{equation}\label{eq:htYt}
	0= \dot{\alpha_t} + d(V_t\rfloor \alpha_t) + V_t\,\rfloor\, d\alpha_t = 
	\beta-\alpha + dh_t + Y_t \,\rfloor\, d\alpha_t, \quad t\in [0,1].
\end{equation}
Contracting this $1$-form with the Reeb vector field $\Theta_t$ and noting that $\Theta_t\rfloor\, d\alpha_t=0$ gives
\begin{equation}\label{eq:PDE}
	 \Theta_t\rfloor\, dh_t= \Theta_t(h_t)  = \Theta_t\rfloor (\alpha-\beta), \quad t\in[0,1].
\end{equation}
%
%
This is a $1$-parameter family of linear holomorphic partial differential equations for 
the functions $h_t$. Note that the contact plane field $\ker\alpha_t$ is tangent to the hypersurface 
$\Sigma=\{z=0\}$ along $R\times \{0\}^{2n}$. Since $\Theta_t \,\rfloor\, \alpha_t =1$, $\Sigma$
is noncharacteristic for the Reeb vector field $\Theta_t$ along $R\times\{0\}^{2n}$
for every $t\in[0,1]$. It follows that the equation \eqref{eq:PDE} has a unique local solution $h_t$ satisfying the initial condition
$h_t|_{\Sigma}=0$ for all $t\in[0,1]$. Since the right hand side of  \eqref{eq:PDE} 
vanishes on $R\times \{0\}^{2n}$, the solutions $h_t$ satisfy
\begin{equation}\label{eq:htvanishes}
	h_t(x,0,\ldots,0)=0\ \ \text{and}\ \ dh_t(x,0,\ldots,0)=0\ \ \text{for all $t\in[0,1]$ and $x\in R$.}
\end{equation}
This choice of $h_t$ ensures that the $\Theta_t$-component of the $1$-form $\beta-\alpha + dh_t$ 
vanishes. Since the $2$-form $d\alpha_t$ is nondegenerate on $\ker\alpha_t$, the equation
\eqref{eq:htYt} has a unique holomorphic solution $Y_t$ tangent to $\ker\alpha_t$. 
In view of \eqref{eq:htvanishes} we have
\[
	\beta-\alpha + dh_t=0\quad \text{on}\ \ R \times \{0\}^{2n}.
\] 
Thus, we see from \eqref{eq:htYt} that the vector field $Y_t$ vanishes along $R \times \{0\}^{2n}$, 
and hence so does $V_t$ \eqref{eq:Vt} in view of \eqref{eq:htvanishes}.
It follows that the flow $\phi_t$ of $V_t$ exists for all  $t\in[0,1]$ in some neighborhood 
of $R \times\{0\}^{2n}$ and it satisfies conditions \eqref{eq:flow} and \eqref{eq:initialcond}. 
This reduces $\beta$ to the normal form $\alpha$ \eqref{eq:normal} on a neighborhood of $R\times \{0\}^{2n}$.
\end{proof}


\section{Local analysis near a Legendrian curve} 
\label{sec:local}  


Let $R$ be an open Riemann surface and $\theta$ be a nowhere vanishing holomorphic $1$-form on $R$.
We consider holomorphic Legendrian curves in the manifold $R\times \C^{2n}$
endowed with the contact form \eqref{eq:normal}:
\begin{equation}\label{eq:alpha}
	\alpha= dz - y_1\theta(x_1) - \sum_{i=2}^n y_i dx_i.
\end{equation}
Here, $x_1$ denotes the variable in $R$ and the other coordinates are  Euclidean coordinates on $\C^{2n}$.
By Theorem \ref{th:normal}, this corresponds to the local analysis near an immersed Legendrian
curve $R\to X$ in a complex contact manifold $(X,\xi)$. 
When applying these results in conjuction with Theorem \ref{th:normal}, we consider only
Legendrian curves in an open neighborhood $\Omega \subset R\times \C^{2n}$ of $R\times\{0\}^{2n}$ 
which corresponds to a Darboux patch in $(X,\xi)$. The following lemma 
is seen by an obvious calculation; see \cite[Lemma 3.1]{AlarconForstnericLopez2017CM}
for the case $R=\C$.

%
%
\begin{lemma}\label{lem:approximate}
Let $\alpha$ be the contact form \eqref{eq:alpha} on $R\times \C^{2n}$.
Given a holomorphic map $f=(x,y,z) \colon \D\to R\times \C^{2n}$, the map
$\tilde f=(x,y,\tilde z) \colon \D\to R\times \C^{2n}$ with 
\begin{equation}\label{eq:tildeF}
	\tilde z(\zeta) = z(\zeta)-\int_0^\zeta f^*\alpha,\quad \zeta\in \D
\end{equation}
is a holomorphic $\alpha$-Legendrian disc satisfying 
\[
	||\tilde z -z||_{0,\D}\le \sup_{|\zeta|<1} \left| \int_0^\zeta f^*\alpha\right|. 
\]
The same is true if $\D$ is replaced by a bordered Riemann surface $M$ 
provided that the holomorphic $1$-form $f^*\alpha$ has vanishing periods over all closed curves in $M$.
\end{lemma}


\begin{proof}[Proof of Theorem \ref{th:embedded}] 
Let $f\colon R\to X$ be a holomorphic Legendrian immersion from an open Riemann surface $R$
to a complex contact manifold $(X,\xi)$ of dimension $2n+1\ge 3$. It suffices to show that for every compact smoothly 
bounded domain $M\subset R$ the restriction $f|_M\colon M\to X$ can be approximated 
in the $\Cscr^1(M)$-norm by a Legendrian embedding $\tilde f\colon M\to X$ of class 
$\Ascr^1(M)=\Cscr^1(M)\cap \Oscr(\mathring M)$.
Fix such a domain $M$. After shrinking $R$ around $M$, Theorem \ref{th:normal}
provides a holomorphic immersion $F\colon Y=R\times r\B^{2n}\to X$, where $\B^{2n}$ is the unit ball
in $\C^{2n}$ and $r>0$, such that $F|_{R\times \{0\}^{2n}}=f$ and the contact structure $F^*\xi$ 
is given by the $1$-form $\alpha$ \eqref{eq:alpha}. Note that the non-isotropic dilation
\begin{equation}\label{eq:dilation}
	z\mapsto t^2 z, \ \ y_1\to t^2 y_1;\quad x_j\mapsto tx_j,\ \ y_j\mapsto t y_j \ \ \text{for}\ j=2,\ldots,n
\end{equation}
for $t\in\C^*$ preserves the contact structure \eqref{eq:alpha}, so we may assume that $r=1$.

For simplicity of notation, we present the proof in the case $n=1$, 
so $\dim X=3$, $Y=R\times\B^2$ and 
\[
	\alpha= dz - y\,\theta(x), \quad x\in M,\ (y,z)\in \B^2.
\]
Let $\Sigma$ be the complex subvariety of $Y\times Y$ defined by
\begin{equation}\label{eq:Sigmaprime}
	\Sigma=\bigl\{\bigl((x,y,z),(x',y',z')\bigr)\in Y\times Y : F(x,y,z)=F(x',y',z')\bigr\}
	= \Delta_Y \cup\Sigma',
\end{equation}
where $\Delta_Y$ denotes the diagonal of $Y\times Y$ and $\Sigma'$ is the union of the irreducible 
components of $\Sigma$ disjoint from $\Delta_Y$.
Since $F$ is an immersion, we have $\dim \Sigma'=3$ or else $\Sigma'=\varnothing$; in the latter case,
$F$ (and hence $f$) is an embedding and there is nothing to prove.

Let $M\ni x\mapsto g_0(x)=(x,0,0)\in Y$ denote the inclusion map $M\hra M\times \{0\}^2 \subset Y$. 
Since $f$ is an immersion, there is an open neighborhood $U\subset M\times M$ of the diagonal 
$\Delta_M$ such that $\overline U\cap \Sigma'=\varnothing$. To prove the theorem, we must find 
arbitrarily close to $g_0$ in the $\Cscr^1(M)$-norm a holomorphic 
$\alpha$-Legendrian map $g\colon M\to Y$ such that the associated map
\[
	g^2\colon M\times M\to Y\times Y,\quad g^2(x,x')=(g(x),g(x'))\ \ (x,x' \in M)
\]
satisfies the condition
\begin{equation}\label{eq:avoiding}
	g(M\times M\setminus U)\cap \Sigma'=\varnothing.
\end{equation}
Assuming that $g$ is close enough to $g_0$, this condition ensures that 
$\tilde f = F\circ g\colon M\to X$ is a holomorphic Legendrian embedding which approximates 
the initial Legendrian immersion $f\colon M\to X$. Indeed, \eqref{eq:avoiding} implies that 
$\tilde f(x)\ne \tilde f(x')$ for all $(x,x')\in M\times M\setminus U$, while
for $(x,x')\in \overline U\setminus \Delta_M$ the same holds provided that $g$ is close to $g_0$
in the $\Cscr^1(M)$-norm.

To find Legendrian maps $g$ satisfying \eqref{eq:avoiding}, we apply the transversality method
together with the technique of controlling the periods. The argument 
is similar to the one in \cite[Lemma 4.4]{AlarconForstnericLopez2017CM} in the case when $R=\C$, $Y=\C^3$
and $\alpha=dz - y\, dx$. It suffices to construct a holomorphic map $H\colon M \times r_0\B^N \to Y$,
where $\B^N$ is the unit ball in $\C^N$ for some big  integer $N\in\N$ and $r_0>0$,  
satisfying the following conditions: 
\begin{itemize}
\item[\rm (a)] $H(\cdotp,0)=g_0$ is the inclusion map $M\hra M\times \{0\}^2 \subset Y$,
\vspace{1mm}
\item[\rm (b)] $H(\cdotp,\xi)\colon M\to Y$ is a holomorphic Legendrian immersion for every $\xi \in r_0\B^N$, and
\vspace{1mm}
\item[\rm (c)]  the map $H^2 \colon M\times M \times  r_0\B^N \to Y\times Y$, defined by 
\[
	H^2(x,x',\xi) = \bigl(H(x,\xi),H(x',\xi)\bigr), \quad x,\, x' \in M, \  \xi \in r_0\B^N,
\]
is a {\em submersive family  of maps} on $M\times M\setminus U$, in the sense that  
\[
	\di_\xi|_{\xi=0} \, H^2(x,x',\xi)  \colon \C^N \to T_x Y \oplus T_{x'}Y \ \ 
	\text{is surjective for every $(x,x')\in M\times M\setminus U$}. 
\]
\end{itemize}

Assume for a moment that such $H$ exists. By compactness of $M\times M \setminus U$, 
it follows from (c) that the partial differential $\di_\xi H^2$ is surjective 
on $(M\times M\setminus U) \times r\B^N$ for some $0<r \le r_0$. For such $r$, the map 
$H^2 \colon (M\times M\setminus U)\times r\B^N \to Y\times Y$ 
is transverse to any complex subvariety of $Y\times Y$, in particular, to $\Sigma'$ (see  \eqref{eq:Sigmaprime}).
It follows that for a generic choice of $\xi\in r\B^N$  the map 
$H^2(\cdotp,\cdotp,\xi)\colon M\times M\to Y\times Y$ is transverse to $\Sigma'$ on 
$M\times M\setminus U$, and hence it does not intersect $\Sigma'$ by dimension reasons. 
(See Abraham \cite{Abraham1963TAMS} or \cite[Section  7.8]{Forstneric2011} 
for the details of this argument.) Choosing $\xi$ sufficiently close to $0\in\C^N$ gives a holomorphic 
$\alpha$-Legendrian embedding $g = H(\cdotp,\xi)\colon M \to Y$ close to $g_0$ 
in the $\Cscr^1(M)$-norm, so $F\circ g\colon M\to X$ is a $\xi$-Legendrian embedding
as explained above.

It remains to explain the construction of $H$.
It suffices to find for any given pair of points $(p,q)\in M\times M\setminus U$ 
a holomorphic spray $H$ as above, with $N=3$, such that 
\begin{equation}\label{eq:specialH}
	H(p,\xi)=(p,0,0) \ \ \text{for all $\xi$}, \ \ \text{and}\ \ 
	\di_\xi|_{\xi=0} \, H(q,\xi)  \colon \C^3 \to T_{q} Y\ \ \text{is an isomorphism}.
\end{equation}
Since submersivity of the differential is an open condition and $M\times M\setminus U$ is 
compact, we then obtain a spray $H$ satisfying conditions (a)--(c) by composing finitely 
many sprays satisfying \eqref{eq:specialH}  (cf.\  \cite[proof of Theorem 2.4]{AlarconForstneric2014IM}).

Let $V$ be a nowhere vanishing holomorphic vector field on $R$,
and let $\phi_t(x)$ denote the flow of $V$ with the initial
condition $\phi_0(x)=x$. Every $h\in \Ascr^1(M)$ sufficiently close to the zero function
determines a holomorphic map $\phi[h]\colon M\to R$ defined by
\[
	\phi[h](x) = \phi_{h(x)}(x),\quad x\in M.
\]
Note that $\phi[0]$ is the identity map, and the Taylor expansion 
in any local coordinate on $R$ is 
\[
	\phi[h](x) = x+ h(x) V(x) + O(|h(x)|^2).
\]

Fix a pair of distinct points $p\ne q$  in $M$. Choose a smooth embedded arc $E \subset M$ 
connecting $p$ to $q$. Let $C_1,\ldots,C_\ell\subset \mathring M$ be closed curves forming 
a basis of the homology group $H_1(M;\Z)=\Z^\ell$ such that 
$\big(\bigcup_{k=1}^\ell C_k\big) \cap E = \varnothing$ and the compact set 
$\big(\bigcup_{k=1}^\ell C_k\big) \cup E$ is Runge in $M$. Let 
$\Pcal=(\Pcal_1,\ldots,\Pcal_\ell)$ be the period map with the components 
\begin{equation}\label{eq:period}
    \Pcal_j(h_1,h_2)=\int_{C_j} h_2\, \phi[h_1]^* \theta,\quad\  j=1,\ldots,\ell,
\end{equation}
defined for all $h_1\in \Ascr^1(M)$ close enough to $0$ and all $h_2\in \Ascr^1(M)$.
Here, $\phi[h_1]^* \theta$ denotes the pull-back of the $1$-form $\theta$ by the map $\phi[h_1]\colon M\to R$.

Let $\mu > 0$ be a small number whose value will be determined  later. 
Choose holomorphic functions $g_1,\ldots, g_\ell, h_1,h_2 \in \Oscr(M)$ satisfying 
the following conditions:
\begin{itemize}
\item[\rm (i)] $\int_{C_j} g_k \, \theta = \delta_{j,k}\ \ \text{for all}\ \ j,k=1,\ldots,\ell$
(here, $\delta_{j,k}$ is the Kronecker symbol);
\vspace{1mm}
\item[\rm (ii)] $g_k(p)=0$ and $|g_k(x)| < 1$ for all $x \in E$ and $k=1,\ldots,\ell$;
\vspace{1mm}
\item[\rm (iii)] $h_1(p)=0$, $h_1(q)=1$, $h_2(p)=h_2(q)=0$;
\vspace{1mm}
\item[\rm (iv)]  $\int_E h_2\, \theta =-1$;
\vspace{1mm}
\item[\rm (v)]  $|h_j(x)| < \mu$ for all $x\in \bigcup_{k=1}^\ell C_k$ and $j=1,2$.
\end{itemize}
Functions with these properties are easily found by first constructing suitable smooth 
functions on the curves $\big(\bigcup_{k=1}^\ell C_k\big) \cup E$ and  applying 
Mergelyan's approximation theorem.
Let $\xi=(\xi_1,\xi_2,\xi_3)\in \C^3$ and $\zeta=(\zeta_1,\ldots,\zeta_\ell)\in \C^\ell$. Consider the function
\begin{equation}
\label{eq:spray-y}
	\tilde y(x,\xi_2,\xi_3,\zeta) = \xi_2 h_2(x) + \xi_3 h_1(x) +  \sum_{k=1}^\ell \zeta_k\, g_k(x), \quad x\in M.
\end{equation}
Note that $\tilde y(x,0,0,0)=0$ for all $x\in M$ and $\tilde y(p,\xi_2,\xi_3,\zeta)=0$ for all $(\xi_2,\xi_3,\zeta)$. 
Condition (i) implies that
\[
	\Pcal_j\bigl(0,\tilde y(\cdotp,0,0,\zeta)\bigr) = \int_{C_j} \sum_{k=1}^\ell \zeta_k\, g_k \, \theta =\zeta_j
\]
for $j=1,\ldots,\ell$, and hence
\begin{equation}\label{eq:der-id}
	\frac{\di}{\di\zeta}\bigg|_{\zeta=0} \Pcal\bigl(0,\tilde y(\cdotp,0,0,\zeta)\bigr) \colon \C^\ell \lra \C^\ell
	\quad \text{is the identity map}.
\end{equation}
The implicit function theorem then shows that the period vanishing equation 
\begin{equation}\label{eq:periodvanishing}
	\Pcal\bigl(\xi_1 h_1, \tilde y(\cdotp,\xi_2,\xi_3,\zeta)\bigr) = 
	\left( \int_{C_j} \tilde y(\cdotp,\xi_2,\xi_3,\zeta)\, \phi[\xi_1 h_1]^* \theta \right)_{j=1,\ldots,\ell}  =0 	
\end{equation}
(which trivially holds at $\xi=0$ and $\zeta=0$)  
can be solved on $\zeta=\zeta(\xi)$ in a neighborhood of $0\in \C^3$, with $\zeta(0)=0$. 
Differentiation of \eqref{eq:periodvanishing} on $\xi$ at $\xi=0$ 
and \eqref{eq:der-id} give
\[
	\zeta'(0) = - \frac{\di}{\di\xi}\bigg|_{\xi=0} \Pcal\bigl(\xi_1 h_1, \tilde y(\cdotp,\xi_2,\xi_3,0) \bigr).
\]
We claim that
\begin{equation}\label{eq:estimatedrho}
	|\zeta'(0)| = O(\mu).
\end{equation}
To see this, note that
\[
	\Pcal_j\bigl(\xi_1 h_1, \tilde y(\cdotp,\xi_2,\xi_3,0) \bigr) = 
	\int_{C_j} (\xi_2 h_2 + \xi_3 h_1) \phi[\xi_1 h_1]^* \theta. 
\]
The partial derivatives on $\xi_2$ and $\xi_3$ equal
$\int_{C_j} h_i \theta$ $(i=1,2)$ which are of size  $O(\mu)$ by condition (v),
while the partial derivative on $\xi_1$ vanishes at $\xi=0$; this proves \eqref{eq:estimatedrho}.

Define the holomorphic spray $\tilde z(\cdotp,\xi)$ with the core $\tilde z(\cdotp,0)=0$ on $M$ by
\begin{equation}\label{eq:tildez}
	\tilde z(x,\xi)= - \int_p^{x} \tilde y(\cdotp,\xi_2,\xi_3,\zeta(\xi))\,  \phi[\xi_1 h_1]^* \theta,\quad x\in M.
\end{equation}
By the choice of $\zeta(\xi)$ the integral is independent of the choice of the path. It follows that
the holomorphic map  $H(\cdotp,\xi)\colon M\to Y$ defined by
\[
	H(x,\xi) = \bigl( \phi[\xi_1 h_1](x), \tilde y(x,\xi_2,\xi_3,\zeta(\xi)), \tilde z(x,\xi) \bigr), \quad x\in M 
\] 
is $\alpha$-Legendrian for every $\xi\in r\B^3$ and is also holomorphic with respect to $\xi$.
Obviously, $H$ satisfies the first condition in \eqref{eq:specialH}. We have
\begin{eqnarray}
	\frac{\di}{\di \xi}\bigg|_{\xi=0} \phi[\xi_1 h_1](q) &=& (V(q),0,0)    \label{eq:xxi} \\
	\frac{\di}{\di \xi}\bigg|_{\xi=0}   \tilde y(q,\xi_2,\xi_3,\zeta(\xi))  &=& (0,0,1)+ O(\mu),   \label{eq:yxi} \\
	\frac{\di}{\di \xi_2}\bigg|_{\xi=0}\, \tilde z(q,\xi)    &=& -\int_E   h_2\theta + O(\mu) 
	= 1+ O(\mu).   \label{eq:zxi}
\end{eqnarray}
The equation \eqref{eq:xxi} is immediate from the definition of the flow and the condition $h_1(q)=1$
(see (iii)), \eqref{eq:yxi} is obvious from the definition of $\tilde y$ and conditions
$h_1(q)=1$, $h_2(q)=0$ (see (iii)), and in \eqref{eq:zxi} we used condition (iv) on $h_2$.
The error terms $O(\mu)$ in \eqref{eq:yxi} and \eqref{eq:zxi} come from the contribution by 
$\zeta'(0)$ (see \eqref{eq:estimatedrho}). By choosing the constant $\mu>0$  small enough, 
we see from \eqref{eq:xxi}, \eqref{eq:yxi} and \eqref{eq:zxi} that $H$ also satisfies the second condition in 
\eqref{eq:specialH}. This completes the proof of Theorem \ref{th:embedded}.
\end{proof}

\begin{remark}[Deformation theory of Legendrian curves] 
\label{remark:deformation}
Let $f_0\colon R\to X$ be an immersed holomorphic Legendrian curve, and let
$M$ be a compact smoothly bounded domain in $R$. The proof of Theorem \ref{th:embedded} 
then shows that the space of all Legendrian immersions $f\colon M\to X$ of class $\Ascr^1(M)$ which are
close to $f_0|_{M}$ is a local complex Banach manifold.
Indeed, in a Darboux chart around $f_0$ provided by Theorem \ref{th:normal},
we consider the space of small perturbations of class $\Ascr^1(M)$ of all 
components of $f_0$ except the $z$-component; clearly this is an open set in a
complex Banach space. We have seen in the proof of Theorem \ref{th:embedded}
that the period vanishing condition for the $1$-form $\sum_{j=1}^n x_j dy_j$
is of maximal rank, and hence it defines a local complex Banach submanifold. 
In view of Lemma \ref{lem:approximate}, this submanifold parametrizes the space of all
small Legendrian perturbations of $f_0|_M$. For the details in a related case
of directed holomorphic immersions, we refer to \cite[Theorem 2.3]{AlarconForstneric2014IM}.
\end{remark}


\section{Proof of Theorem \ref{th:main}} 
\label{sec:proof-main}

Theorem \ref{th:main} follows by an inductive application of the following lemma.

\begin{lemma}
\label{lem:elongate}
Let $(X,\xi)$ be a complex contact manifold, and let $\dist_X$ be a distance
function on $X$ induced by a Riemannian metric. Given a compact bordered
Riemann surface $M$, a point $u_0\in \mathring M$,
a holomorphic Legendrian immersion $f_0\colon M\to X$,
and a number $\epsilon>0$, there exists a holomorphic Legendrian 
embedding $f\colon M\hra X$ such that 
\begin{equation}\label{eq:est-f1}
	\sup_{u\in M}\dist_X(f(u),f_0(u))<\epsilon
\end{equation}	
and
\begin{equation}\label{eq:est-f2}
	\inf\bigl\{\length_X(f(\gamma)) : 
	\text{$\gamma\subset M$ is a path connecting $u_0$ and $bM$}\bigr\}> \frac1{\epsilon}.
\end{equation}
\end{lemma}

By using this lemma, it is a trivial matter to construct a sequence of 
holomorphic Legendrian embeddings $f_j\colon M \hra X$ converging uniformly on $M$
to a topological embedding $f=\lim_{j\to\infty} f_j\colon M\hra X$ which is as close as desired
to $f_0$ uniformly on $M$ and whose restriction to $\mathring M$ is a complete holomorphic 
Legendrian embedding. We refer to \cite[proof of Theorem 1.1]{AlarconDrinovecForstnericLopez2015PLMS} 
for a detailed explanation in a similar geometric context.

In the standard case when $X=\C^{2n+1}$ is endowed with the standard contact structure, 
Lemma \ref{lem:elongate} coincides with \cite[Lemma 6.5]{AlarconForstnericLopez2017CM}.
Although the latter lemma is stated only for the case when $M$ is the disc, it is explained there
how the proof extends to any compact bordered Riemann surface.
In the case at hand, we shall use \cite[Lemma 6.5]{AlarconForstnericLopez2017CM} 
together with the existence of Darboux charts furnished by Theorem \ref{th:normal}.

For simplicity of notation we assume that $\dim X=3$; the same proof applies in general.
We also assume without loss of generality that $M$ is a compact smoothly bounded domain in an open 
Riemann surface, $R$, and that $f_0$ extends to a holomorphic Legendrian immersion $R\to (X,\xi)$. 
By Theorem \ref{th:embedded} we may further assume, up to shrinking $R$ around $M$ if necessary, 
that $f_0\colon R\hra (X,\xi)$ is a holomorphic Legendrian embedding.  

Choose a holomorphic immersion $x_0\colon R\to \C$ (see \cite{GunningNarasimhan1967MA}).
After shrinking $R$ around $M$ if necessary, Theorem \ref{th:normal}
provides a holomorphic embedding $F\colon Y=R\times \B^2 \hra X$
such that the contact structure $F^*\xi$ on $Y$ is determined by the $1$-form
$\alpha=dz - y dx_0$. More precisely, letting $\alpha_0=dz - ydx$ denote the standard 
contact form on $\C^3$ and $G\colon Y \to \C^3$ the immersion 
$G(u,y,z)=(x_0(u),y,z)$, we have $\alpha=G^*\alpha_0$. 
Note that $g_0=G(\cdotp,0,0)=(x_0,0,0)\colon R\to \C^3$ is an $\alpha_0$-Legendrian immersion.

Fix a Riemannian metric on $R$ with the distance function $\dist_R$.
By using the Euclidean metric on $\C^2$ we thus get a metric and 
a distance function on $Y=R\times \B^2$.

Let $r_0>0$ be the radius of injectivity of the immersion $x_0\colon R\to\C$ on $M$. 
This means that for every point $u\in M$, $x_0$ maps the geodesic disc $D(u,r_0)\subset R$ around $u$ 
bijectively onto its image $x_0(D(u,r_0)) \subset \C$. Let $r_0'>0$ be chosen such that the Euclidean
disc $\D(x_0(u),r_0') \subset \C$ lies inside $x_0(D(u,r_0))$ for every $u\in M$; recall that $M$ is compact. 

\begin{lemma}\label{lem:observation}
There is a constant $C\ge 1$ such that for every function 
$x\colon M\to \C$ which is uniformly $r_0'$-close to $x_0\colon M\to\C$ on $M$
there exists a unique map $\phi \colon M\to R$ which is $r_0$-close to the identity 
on $M$ such that $x=x_0 \circ \phi$ and 
\begin{equation}
\label{eq:observation}
	C^{-1} \|x-x_0\|_M   \le \dist_M(\phi,\Id_M) \le C \|x-x_0\|_M. 
\end{equation}
If $x$ is holomorphic on $M$ then so is $\phi$.
\end{lemma}

\begin{proof}
For every $u\in M$ we have $x(u) \in x_0(D(u,r_0))$, and hence there is a unique point 
$\phi(u)\in D(u,r_0)$ such that $x_0(\phi(u))=x(u)$. Clearly, this determines the map $\phi$
with the stated properties. The estimate \eqref{eq:observation}
follows from the inverse mapping theorem.
\end{proof}

Since $F\colon Y\to X$ is a biholomorphism onto the domain $F(Y)\subset X$, there is a number $\eta>0$ such that
for every immersion $h\colon M\to Y$ satisfying 
\begin{equation}\label{eq:est-h1}
	\dist_Y(h,\Id_M)<\eta
\end{equation}
and
\begin{equation}\label{eq:est-h2}
	\inf\bigl\{\length_Y(h(\gamma)) : 
	\text{$\gamma\subset M$ is a path connecting $u_0$ and $bM$}\bigr\} > \frac1{\eta},
\end{equation}
the immersion $f=F\circ h\colon M\to X$ satisfies the estimates \eqref{eq:est-f1} and  \eqref{eq:est-f2}.
Furthermore, if $h$ is a holomorphic $\alpha$-Legendrian embedding, then $f=F\circ h\colon M\to X$
is a holomorphic $\xi$-Legendrian embedding.
Hence, $f$ satisfies the conclusion of Lemma \ref{lem:elongate}.

It remains to find $h$. Pick a number $\delta>0$. By \cite[Lemma 6.5]{AlarconForstnericLopez2017CM}
there is a holomorhic $\alpha_0$-Legendrian immersion 
$g_1=(x_1,y_1,z_1)\colon M\to \C^3$ which is uniformly $\delta$-close to 
$g_0=(x_0,0,0)$ on $M$ and satisfies
\[
	\inf\bigl\{\length(g_1(\gamma)) : 
	\text{$\gamma\subset M$ is a path connecting $u_0$ and $bM$}\bigr\}>\frac1{\delta}.
\] 
In particular, the first component $x_1$ of $g_1$ is $\delta$-close to $x_0$ on $M$.
Choosing $\delta>0$ small enough, Lemma \ref{lem:observation} ensures that $x_1=x_0\circ \phi$ 
for some holomorphic map $\phi\colon M\to \R$; furthermore, the map 
$h=(\phi,y_1,z_1) \colon M\to Y=R\times \B^2$ is a holomorphic $\alpha$-Legendrian 
immersion satisfying \eqref{eq:est-h1} and \eqref{eq:est-h2}.
By Theorem \ref{th:embedded} we can choose $h$ to be an embedding.
This completes the proof.

%
%

\section{Contact neighborhoods of isotropic Stein submanifolds}\label{sec:CNT}

In this section, we consider isotropic complex submanifolds $M$ of higher dimension in a 
complex contact manifold $(X^{2n+1},\xi)$, and we prove a contact neighborhood theorem in the case when
$M$ is a Stein submanifold; see Theorem \ref{th:contactnbd}. In the special case when $M$ is Stein and 
contractible, we construct a Darboux chart around it; see Theorem \ref{th:Darboux2}. 
For simplicity of exposition we assume that $M$ is embedded, although the analogous results 
also apply to immersed submanifolds.

We begin by recalling a few basic facts; see e.g.\ \cite{LeBrun1995IJM,LeBrunSalamon1994IM}. The quotient projection 
\[ 
	 TX\stackrel{\alpha}{\longrightarrow} \nu:=TX/\xi
\] 
onto the {\em normal line bundle} $\nu$ of the contact subbundle $\xi$ is a nowhere vanishing
holomorphic $1$-form $\alpha$ on $X$ with values in $\nu$ such that  $\xi=\ker\alpha$. 
The differential $d\alpha$ defines a holomorphic section of $\Lambda^2(\xi^*)\otimes \nu$, and  
$\alpha\wedge (d\alpha)^n\ne 0$ is a nowhere vanishing section of the line bundle $K_X\otimes \nu^{\otimes(n+1)}$,
where $K_X=\Lambda^{2n+1}(T^*X)$ is the canonical line bundle of $X$.
This provides a holomorphic line bundle isomorphism 
\begin{equation}\label{eq:nu}
	 \nu^{\otimes(n+1)} \cong K_X^{-1}=\Lambda^{2n+1}(TX)
\end{equation}
between the $(n+1)$-st tensor power of the normal bundle $\nu$ and the anticanonical bundle $ K_X^{-1}$ of $X$. 
On any open subset of $X$ over which the bundle $\nu$ is trivial, we may consider $\alpha$ as a scalar-valued $1$-form 
determined up to a nonvanishing holomorphic factor. The condition $\alpha\wedge (d\alpha)^n\ne 0$ implies that 
\[
	\omega:=d\alpha|_\xi 
\]
is a holomorphic symplectic form on the contact bundle $\xi$ which is determined up to a nonvanishing factor
(since $d(f\alpha)=fd\alpha + df\wedge \alpha=fd\alpha$ on $\xi=\ker\alpha$), so the pair 
$(\xi,\omega)$ is a {\em conformal symplectic holomorphic vector bundle} over $X$. 
The restriction of $\omega$ to any fibre $\xi_x$ $(x\in X)$ 
is a nondegenerate skew-symmetric bilinear form $\omega_x\colon \xi_x\times \xi_x\to\C$. 

Given an $\R$-linear subspace $U$ of $\xi_x$, we denote by 
$U^\perp$ its $\omega_x$-orthogonal complement:
\[
	U^\perp = \bigcap_{u\in U} \ker \omega_x(u,\cdotp) \subset \xi_x.
\]
Note that $U^\perp$ depends only on the conformal class of $\alpha$, and hence
only on the contact structure $\xi$.
Since $\omega_x(u,\cdotp)\colon \xi_x\to\C$ is a $\C$-linear form, $U^\perp$ is a complex subspace of $\xi_x$.
We say that $U$ is {\em $\omega_x$-isotropic} if $U\subset U^\perp$; equivalently, 
$\omega_x(u,v)=0$ for any pair of vectors $u,v\in U$. Since $U^\perp$ is complex,
it follows that $U^\C:= \Span_\C(U) \subset U^\perp$. Since the $2$-form $\omega_x$ on $\xi_x$ is nondegenerate, 
we have the following dimension formula (see \cite[Sec.\ 2.4]{Geiges2008}):
\begin{equation}\label{eq:dimformula}
	\dim_\C U^\C + \dim_\C U^\perp=\dim_\C\xi_x=2n.
\end{equation}
It follows that any real isotropic subspace $U$ of $\xi_x$ satisfies 
$\dim_\R U\le 2n$, and we have $\dim_\R U=2n$ if and only if $U=U^\C=U^\perp$;
such  $U$ is said to be {\em $\omega_x$-Lagrangian}. In particular, we see that every 
Lagrangian subspace of $\xi_x$ is complex.

These notions and observation extend in an obvious way to smooth submanifolds $M$ of $X$.
Thus, $M$ is isotropic with respect to the contact structure $\xi=\ker\alpha$ if $T_x M\subset \xi_x$ holds for all $x\in M$; 
equivalently, if $\alpha|_{TM}=0$. This implies that $\omega=d\alpha|_\xi$ also vanishes on the tangent bundle 
$TM$, so $T_x M$ is $\omega_x$-isotropic for every $x\in M$ and therefore $\dim_\R M\le 2n$. A $\xi$-isotropic submanifold $M$
of $X$ is said to be {\em Legendrian} if it has maximal real dimension $2n$. We have seen above that for any such submanifold,
the tangent space $T_x M$ at each point $x\in M$ is a complex linear subspace of $\xi_x$ (satisfying $T_x M=(T_x M)^\perp$);
it follows that $M$ is a complex submanifold of $X$. This observation seems worthwhile recording.
(The special case for $\dim X=3$ was observed in \cite[Proposition 1.5]{ForstnericLarusson2017MZ}.)

\begin{lemma}\label{lem:Legendrian}
Let $(X,\xi)$ be a complex contact manifold. Every smooth Legendrian submanifold $M$ of $(X,\xi)$ 
is a complex submanifold of $X$.
\end{lemma}

Since a Stein manifold does not admit any compact 
complex submanifolds of positive dimension, the lemma implies

\begin{corollary}
A Stein contact manifold does not have any compact smooth Legendrian submanifolds.
\end{corollary}

We now introduce the relevant decomposition of the complex normal bundle 
\[
	N(M,X)=TX|_M/TM
\]
of an isotropic complex submanifold $M\subset X$. Assume that $\dim X=2n+1\ge 3$ and $\dim M=m \in\{1,\ldots,n\}$. 
Recall that $\nu=TX/\xi$. We have 
\begin{equation}\label{eq:normalbundle}
	N(M,X) = \nu|_M  \oplus (\xi|_M/TM^\perp) \oplus (TM^\perp /TM).
\end{equation}
(Compare with \cite[(2.6), p.\ 69]{Geiges2008} for the analogous decomposition in the smooth case.)
If $M$ is a Stein manifold, then the component bundles in \eqref{eq:normalbundle} can be embedded as holomorphic
vector subbundles of $N(M,X)$, and the latter is a holomorphic subbundle of 
the restricted tangent bundle $TX|_M=TM\oplus N(M,X)$.
In particular, \eqref{eq:normalbundle} can be seen as an internal direct sum of holomorphic vector subbundles.
Since the rank of the bundle $TM^\perp$ equals $2n-m$ in view of the dimension formula
\eqref{eq:dimformula}, the summands on the right hand side of \eqref{eq:normalbundle}
have ranks $1,m$, and $2(n-m)$, respectively. By \cite[Lemma 2.5.4]{Geiges2008}, the bundle
$\xi|_M/TM^\perp$ is isomorphic to the cotangent bundle $T^*M$ via the 
bundle isomorphism 
\begin{equation}\label{eq:Tstar}
	\xi|_M/ TM^\perp \to T^*M,\quad 
	\xi_x \ni v \mapsto i_v \omega|_{TM}\in T_x^*M \ \ (x\in M).
\end{equation}
Thus, assuming that the normal bundle $\nu=TX/\xi$ of $\xi$ is trivial over the submanifold $M$, 
the only bundle in the decomposition \eqref{eq:normalbundle}
which depends on the isotropic embedding $M\hra X$ is the rank $2(n-m)$
{\em conformal symplectic normal bundle}
\begin{equation}\label{eq:CSN}
	CSN(M,X) := TM^\perp /TM.
\end{equation}
This brings us to the following {\em holomorphic contact neighborhood theorem}, analogous to 
the corresponding result in smooth contact geometry (see e.g.\ \cite[Theorem 2.5.8, p.\ 71]{Geiges2008}).

%
%
%
\begin{theorem}\label{th:contactnbd}
Let $(X_i,\xi_i)$ $(i=0,1)$ be complex contact manifolds of dimension $2n+1\ge 3$ with locally closed
isotropic Stein submanifolds $M_i \subset X_i$. Suppose that $\nu_i=TX_i/\xi_i$ is trivial over $M_i$ for $i=0,1$, and  
there is a holomorphic vector bundle isomorphism of conformal symplectic normal bundles 
$\Phi\colon CSN(M_0,X_0)\to CSN(M_1,X_1)$ that covers a biholomorphism $f \colon M_0\to M_1$. 
Then $f$ extends to a holomorphic contactomorphism $F \colon \Ncal(M_0)\to \Ncal(M_1)$ 
of suitable Stein neighbourhoods of $M_i$ in $X_i$ such that 
\begin{equation}\label{eq:TF}
	TF|_{CSN(M_0,X_0)}=\Phi.
\end{equation}
%
%
%
\end{theorem}

Theorem \ref{th:cor:contactnbd} in the introduction 
follows from Theorem \ref{th:contactnbd} by noting that for a Legendrian submanifold 
$M$ of $(X,\xi)$ (that is, an isotropic submanifold of maximal dimension)  
the conformal normal bundle \eqref{eq:CSN} has rank zero, and hence the condition
regarding $\Phi$ in Theorem \ref{th:contactnbd} is void.

\begin{proof}
The proof is analogous to that of \cite[Theorem 2.5.8]{Geiges2008}.
We also take into account \cite[Remark 2.5.12]{Geiges2008} and get
the sharper statement as in \cite[Theorem 6.2.2, p.\ 294]{Geiges2008}.

The assumption that the line bundle $\nu_i=TX_i/\xi_i$ is trivial on $M_i$ implies 
that $\xi_i=\ker\alpha_i$ for a holomorphic scalar-valued contact form $\alpha_i$
in a neighborhood of $M_i$ in $X_i$ $(i=0,1)$. Hence we can identify $\nu_i|_{M_i}$  
with the line subbundle $\langle R_i\rangle$ of $TX_i|_{M_i}$ spanned by the Reeb vector field $R_i$
of the contact form $\alpha_i$. Let $\omega_i=d\alpha_i|_{\xi_i}$ be the associated
holomorphic symplectic form on the contact subbundle $\xi_i$.
There is a unique  holomorphic isomorphism 
\[
	\Theta \colon N(M_0,X_0) \to N(M_1,X_1)
\]
of normal bundles, covering the biholomorphism $f\colon M_0\to M_1$, which is determined on the 
component subbundles in the decomposition \eqref{eq:normalbundle} as follows:
\begin{itemize}
\item on the line subbundle $\nu_0=\langle R_0\rangle$ we have $\Theta(R_0(x))=R_1(f(x))$, $x\in M_0$;
\vspace{1mm}
\item on $CSN(M_0,X_0)$ we let $\Theta=\Phi\colon CSN(M_0,X_0)\to CSN(M_1,X_1)$ be the
isomorphism in the statement of the theorem;
\vspace{1mm}
\item on $\xi|_{M_0}/ (TM_{0})^\perp_{\omega_0} \cong T^*M_0$ (see \eqref{eq:Tstar}) we let $\Theta=T^*(f^{-1})$,
the cotangent map of the biholomorphism $f^{-1}\colon M_1\to M_0$.
\end{itemize}
Since the submanifolds $M_i\subset X_i$ are locally closed and Stein, the Docquier-Grauert-Siu tubular neighborhood theorem 
(see \cite[Theorem 3.3.3, p.\ 67]{Forstneric2011}) implies that for $i=0,1$ there are an open Stein neighborhood 
$\Omega_i\subset X_i$ of $M_i$ and a biholomorphism $\Psi_i\colon \Omega_i\to \Omega'_i$ onto a 
neighborhood $\Omega'_i$ of the zero section in the normal bundle $N(M_i,X_i)$ such that,
if we identify the zero section with $M_i$ and note the canonical decomposition
\[
	T(N(M_i,X_i))|_{M_i} \cong TM_i \oplus  N(M_i,X_i) = TX_i|_{M_i},
\]
the differential $T\Psi_i|_{M_i}:TX_i|_{M_i}\to T(N(M_i,X_i))|_{M_i}$ is the identity map. 
After suitably shrinking the neighborhoods $\Omega_i\supset M_i$, the composition
\[
	G = \Psi_1^{-1} \circ \Theta \circ \Psi_0 \colon \Omega_0 \to \Omega_1
\] 
is a biholomorphism extending $f\colon M_0\to M_1$ such that the contact forms
$G^*\alpha_1$ and $\alpha_0$ agree on $TX_0|_{M_0}$ along with their differentials.
It follows from the construction that the tangent map $TG$ respects the decomposition
\eqref{eq:normalbundle} of the component bundles; in particular, we have $TG=\Phi$ on $CSN(M_0,X_0)$. 

It remains to correct $G$ to a contactomorphism. This is accomplished by finding a biholomorphic map 
$\psi$ on a neighborhood of $M_0$ in $X_0$ which fixes $M_0$ and satisfies 
\[
	\text{$T\psi=\Id$ on $TX_0|_{M_0}$}\quad \text{and} \quad \psi^*(G^*\alpha_1)=\alpha_0.
\]
The biholomorphism $F=G\circ \psi$ from a neighborhood of $M_0$ in $X_0$ onto a neighborhood of $M_1$ 
in $X_1$ then satisfies $F^*\alpha_1=\alpha_0$ and also \eqref{eq:TF}.

A biholomorphism $\psi$ with these properties is furnished by the following proposition,
applied  with the pair of contact forms $\alpha=\alpha_0$ and $\beta=G^*\alpha_1$ on $X=X_0$,
with the isotropic Stein submanifold $M=M_0\subset X_0$.

%
%
%
%

\begin{proposition}\label{prop:tangent} 
Assume that $(X,\alpha)$ is a complex contact manifold with a locally closed
isotropic Stein submanifold $M\subset X$, and $\beta$ is a holomorphic $1$-form
on a neighborhood of $M$ in $X$ such that $\alpha=\beta$ and $d\alpha=d\beta$
hold on $TX|_M$. Then there exist a neighborhood $\Omega\subset X$ of $M$ and a
biholomorphism $\psi \colon \Omega\to\psi(\Omega)\subset X$, fixing $M$ pointwise,
whose differential agrees with the identity map on $TX|_M$ and satisfies $\psi^*\beta=\alpha$.
\end{proposition}

The proof of the proposition is obtained by a refinement of Moser's method. We shall adjust
\cite[proof of Theorem 6.2.2, p.\ 294]{Geiges2008} to the holomorphic case.

\begin{proof} 
The conditions imply that  
\[
	\alpha_t=(1-t)\alpha + t \beta\quad (t\in[0,1])
\]
is an isotopy of contact 1-forms in a neighborhood of $M$ in $X$ 
such that $\alpha_t$ and $d\alpha_t$ are independent of $t$ on $TX|_M$;
hence, the Reeb vector field $R_t$ of $\alpha_t$ is also independent of $t$ along $M$. 
%
%
Since $M$ is Stein, there is a complex submanifold $\Sigma$ of $X$ containing $M$ such that 
$T\Sigma|_M=(\ker\alpha)|_M$. By shrinking $\Sigma$ around $M$ if necessary,
we may assume that the pullback of $d\alpha_t$ to $\Sigma$ is a holomorphic symplectic form $\omega_t$ 
on $\Sigma$ whose restriction to $T\Sigma|_M$ is independent of $t$. 
%
%
Hence the $2$-form 
\[
	\eta=\omega_1-\omega_0=\dot \omega_t
\]
on $\Sigma$ is closed and it vanishes on $T\Sigma|_M$. Since $M$ is Stein, the generalized Poincar\'e lemma 
(see \cite[Corollary A.4, p.\ 403]{Geiges2008}) gives a holomorphic $1$-form $\zeta$ on 
a neighborhood of $M$ in $\Sigma$ that vanishes to the second order on $M$ and satisfies
\[
	d\zeta=\eta=\dot \omega_t. 
\]
Let $V_t$ be the holomorphic vector field on $\Sigma$
uniquely determined by the equation  
\[
	\zeta+V_t\rfloor \omega_t=0. 
\]
Then $V_t$ vanishes to the second order on $M$. Its flow $\phi_t$ $(t\in [0,1])$ exists in a neighborhood of $M$ in $\Sigma$,
it fixes $M$ pointwise, and it satisfies $T\phi_t=\Id$ on $T\Sigma|_M$. Furthermore, 
\[
	\frac{d}{dt} (\phi_t^*\omega_t) = \phi_t^*(L_{V_t}\omega_t + \dot \omega_t) = 
	\phi_t^*\bigl( d(V_t\rfloor \omega_t) + \dot \omega_t \bigr) = \phi_t^*(-d\zeta + \dot \omega_t)=0
\]
which implies $\phi_t^*\omega_t=\omega_0$ for all $t\in[0,1]$; in particular, $\phi_1^*\omega_1=\omega_0$. 
We extend $\phi_1$ from $\Sigma$ to a biholomorphism 
$\phi$ on a neighborhood of $M$ in $X$ by requiring that it sends flow lines of the Reeb vector field 
$R_0$ to those of $R_1$. Since $R_0=R_1$ on $M$, this gives $T\phi=\Id$ on $TX|_M$. 
This extension satisfies $\phi^*(d\beta)=d\alpha$ on a neighborhood of $M$.  

Replacing $\beta$ by $\phi^*\beta$, we have thus reduced the proof  to the case when 
$\alpha=\beta$ on $TX|_M$ and $d\alpha=d\beta$ holds on a neighborhood of $M$ in $X$. As before, set 
$\alpha_t=(1-t)\alpha + t \beta$ and look for a holomorphic flow $\psi_t$ satisfying 
$\psi_t^*\alpha_t=\alpha_0$ for all $t\in[0,1]$. Since $\dot\alpha_t = \beta-\alpha$ is a closed $1$-form 
that vanishes on $TX|_M$, the generalized Poincar\'e lemma (see \cite[Corollary A.4, p.\ 403]{Geiges2008}) 
can be used as above to obtain a solution satisfying $T\psi_t=\Id$ on $TX|_M$ for every $t\in [0,1]$. 
(Equivalently, the vector field $V_t$ generating the flow $\psi_t$ vanishes to the second order on $M$.) 
The map $\psi=\psi_1$ then satisfies the conclusion of Proposition \ref{prop:tangent}.
For further details, see \cite[proof of Theorem 6.2.2]{Geiges2008}.
\end{proof}
\vspace{-2mm}
This completes the proof of Theorem \ref{th:contactnbd}.
\end{proof}

\begin{remark}
In the smooth case, Proposition \ref{prop:tangent} holds for any smooth submanifold $M$ of $X$. 
In the holomorphic case, the tubular neighborhood theorem, which is used in the
proof of the generalized Poincar\'e lemma, is available only if $M$ is a Stein manifold.
\end{remark}

The following result is an analogue of Theorem \ref{th:normal} in the case when $M$ is a topologically contractible
Stein manifold, embedded as an isotropic submanifold in a complex contact manifold $(X,\xi)$. 
Contractibility of $M$ implies that all holomorphic  vector bundles over $M$ are topologically trivial, and 
hence holomorphically trivial by the Oka-Grauert principle (see \cite[Sec.\ 5.3]{Forstneric2011}). 
This allows us to choose global holomorphic coordinates on the normal bundle $N(M,X)$, 
thereby obtaining the Darboux normal form \eqref{eq:normal3} for 
the contact structure $\xi$ in a tubular neighborhood of $M$ in $X$ 
by following the proof of Theorem \ref{th:normal}. 

%
%
%
%
\begin{theorem}\label{th:Darboux2}
Let $(X,\xi)$ be a  complex contact manifold of dimension $2n+1\ge 3$. Assume that 
$M$ is a contractible Stein manifold of dimension $m$, $\theta_1,\ldots,\theta_m$ are 
holomorphic $1$-forms on $M$ providing a framing of the cotangent bundle $T^*M$,
and $f\colon M\to X$ is a holomorphic  
isotropic immersion.
Then there are a neighborhood $\Omega \subset M\times \C^{2n+1-m}$ of $M\times \{0\}$
and a holomorphic immersion $F\colon \Omega \to X$ (embedding if $f$ is an embedding) such that 
$F|_{M\times \{0\}}=f$ and  the contact structure $F^*\xi$ on $\Omega$ is given by the 
contact $1$-form
\begin{equation}\label{eq:normal3}
	\alpha= dz- \sum_{j=1}^m y_j\theta_j - \sum_{i=m+1}^n y_i dx_i,
\end{equation}
where 
$(x_{m+1},\ldots,x_{n},y_1,\ldots,y_n,z)$ are complex coordinates on $\C^{2n+1-m}$.
\end{theorem}

\begin{proof}
The proof is similar to that of Theorem  \ref{th:normal}, so we only include a brief sketch. 

Let $p=2n-m$. Since $f\colon M\to X$ is isotropic, we have  $m\le n$ and hence $p\ge n$.
As in that proof, we find a Stein open neighborhood $\Omega \subset M\times \C^{p+1}$ of 
$M\times \{0\}^{p+1}$ and a holomorphic immersion $F\colon \Omega \to X$, with $F|_{M\times \{0\}}=f$,
such that $F^*\xi = \ker\beta$, where $\beta$ is a holomorphic $1$-form on $\Omega$
satisfying $\beta \wedge(d\beta)^n \neq 0$ and $M\times \{0\}^{p+1}$ is a
$\beta$-Legendrian submanifold of $M\times\C^{p+1}$. 
Let $x$ denote points in $M$ and $\zeta=(\zeta_0,\zeta_1,\ldots,\zeta_{p})$ 
be complex coordinates on $\C^{p+1}$. Along $M \times \{0\}^{p+1}=\{\zeta=0\}$ we have
$\beta(x)=\sum_{j=0}^{p} a_j(x) d\zeta_j$ $(x\in M)$ 
for some holomorphic functions $a_j\in\Oscr(M)$ without common zeros. The $1$-forms
$\theta_1,\ldots,\theta_m$ on $M$ do not appear in the above expression since $M\times \{0\}^{p+1}$ is 
$\beta$-Legendrian. By the same argument as in the proof of Theorem  \ref{th:normal} we transform
$\beta$ to the $1$-form $d\zeta_{0}$ along $M\times\{0\}^{p+1}$, and
we rename the variable $\zeta_{0}$ by calling it $z$. After dividing $\beta$ with the coefficient of $dz$
(which is nonvanishing in a neighborhood of $M\times\{0\}^{p+1}$) we have
\begin{equation}\label{eq:Taylor1b}
	\beta=dz+ \sum_{i=1}^m \biggl( \sum_{j=1}^{p} a_{i,j}(x) \zeta_j \biggr) \theta_i(x) 
	+ \sum_{j,k=1}^{p} c_{j,k}(x)\zeta_k\, d\zeta_j + \tilde \beta,
\end{equation}
where the $1$-form $\tilde \beta$ contains all terms whose coefficients are of order 
$\ge 2$ in the variables  $\zeta_1,\ldots,\zeta_{p}$ or they contain the $z$ variable; 
these terms disappear in $\beta \wedge(d\beta)^n$ at all points of $M \times \{0\}^{p+1}$.
By a similar argument as in the proof of Theorem  \ref{th:normal}, we 
see that the $m\times p$ matrix of coefficients $A(x)=(a_{i,j}(x))$ in \eqref{eq:Taylor1b} 
has maximal rank $m$ at every point $x\in M$.
(Indeed, if this fails at some point $x_0\in M$, we easily see that $\beta \wedge(d\beta)^n=0$ 
at $(x_0,0)\in M\times \{0\}^{p+1}$.) 
Hence, Lemma \ref{lem:GL} provides a holomorphic change of coordinates 
$(x,z,\zeta)\mapsto (x,z,B(x)\zeta)$, with $B(x)\in GL_{p}(\C)$ for all $x\in M$,
which reduces $\beta$ to 
\begin{equation}\label{eq:Taylor2b}
	\beta=dz - \sum_{i=1}^m \zeta_i \theta_i + \sum_{j,k=1}^{p} c_{j,k}(x)\zeta_k \, d\zeta_j + \tilde \beta.
\end{equation}
%
%
Since $d(\zeta_i\theta_i)=d\zeta_i\wedge \theta_i+\zeta_i d\theta_i=d\zeta_i\wedge \theta_i$ on $M\times \{0\}$, 
the differentials $d\theta_i$ do not contribute to $d\beta$ on $M\times \{0\}^{p+1}$, and 
$(d\beta)^n$ contains the factor $\wedge_{i=1}^m (d\zeta_i\wedge \theta_i)$.
Hence, we can move the terms with $d\zeta_1,\ldots,d\zeta_m$ in the second sum in \eqref{eq:Taylor2b},
as well as all terms containing $\zeta_k d\zeta_j$ for any $j>m$ and $k\le m$, into the remainder $\tilde \beta$
since none of these terms contributes to $\beta \wedge(d\beta)^n$ on $M\times \{0\}^{p+1}$.
%
%
%
Renaming the variables $\zeta_1,\ldots,\zeta_m$ by calling them $y_1,\ldots,y_m$, we thus have
\begin{equation}\label{eq:Taylor3b}
	\beta=dz - \sum_{i=1}^m  y_i\theta_i  + \sum_{j,k=m+1}^{p}  c_{j,k}(x)\zeta_k \, d\zeta_j + \tilde \beta.
\end{equation}
If $m=p$ (in which case the immersion $f\colon M\to X$ is Legendrian), we are done.
Otherwise, we finish the reduction as in the proof of Theorem \ref{th:normal}, changing 
the second sum on the right hand side of \eqref{eq:Taylor3b} to $ - \sum_{i=m+1}^n y_i dx_i$
after having suitably renamed the variables $\zeta_{m+1},\ldots,\zeta_p$.

Once the normalization of $\beta$ along $M\times \{0\}^{p+1}$ has been achieved, 
one completes the proof exactly as before by applying Moser's method, thereby 
removing the remainder $\tilde \beta$ and changing $\beta$ to the normal
form  \eqref{eq:normal3} in a neighborhood of $M\times\{0\}^{p+1}$ in $M\times \C^{p+1}$.
\end{proof}


\subsection*{Acknowledgements}
A.\ Alarc\'on is supported by the Ram\'on y Cajal program of the Spanish Ministry of Economy and Competitiveness
and by the MINECO/FEDER grant no. MTM2014-52368-P, Spain. 
F.\ Forstneri\v c is supported  by the research grants P1-0291 and J1-7256 from 
ARRS, Republic of Slovenia. A part of the work was done while the authors
were visiting the Center for Advanced Study in Oslo. They wish to thank this institution for 
the invitation, partial support (to Forstneri\v c), and excellent working conditions.
The initial version of the paper was prepared while Forstneri\v c was visiting the 
Department of Geometry and Topology of the  University of Granada, Spain, in 
January and February 2017, and the revised version was prepared 
during his visit at the University of Adelaide, Australia, in May 2017. 
He thanks both institutions for the invitation, hospitality and partial support. 

The authors wish to thank an anonymous referee for proposing to include results on contact
neighborhoods of isotropic Stein submanifolds of dimension $>1$. 
They also thank Finnur L\'arusson for a helpful discussion of this subject.


{\bibliographystyle{abbrv} \bibliography{bibAFL}}


\vspace*{0.5cm}
\noindent Antonio Alarc\'{o}n

\noindent Departamento de Geometr\'{\i}a y Topolog\'{\i}a e Instituto de Matem\'aticas (IEMath-GR), Universidad de Granada, Campus de Fuentenueva s/n, E--18071 Granada, Spain

\noindent  e-mail: {\tt alarcon@ugr.es}

\vspace*{0.5cm}
\noindent Franc Forstneri\v c

\noindent Faculty of Mathematics and Physics, University of Ljubljana, Jadranska 19, SI--1000 Ljubljana, Slovenia.

\noindent Institute of Mathematics, Physics and Mechanics, Jadranska 19, SI--1000 Ljubljana, Slovenia.

\noindent e-mail: {\tt franc.forstneric@fmf.uni-lj.si}

\end{document}